\documentclass{article}
\usepackage[utf8]{inputenc}
\usepackage{amsfonts}
\usepackage{epigraph}

\usepackage{amsfonts}
\usepackage{amssymb}
\usepackage{amsmath}
\usepackage{amssymb}
\usepackage{float}

\usepackage{amsthm}
\usepackage[english]{babel}

\usepackage{fullpage}
\usepackage{wrapfig}

\newtheorem{theorem}{Theorem}
\newtheorem{lemma}{Lemma}
\newtheorem{corollary}{Corollary}

\DeclareMathOperator{\disc}{disc}
\DeclareMathOperator{\supp}{supp}

\def\lc{\left\lceil}   
\def\rc{\right\rceil}

\title{Coloring general Kneser graphs and hypergraphs via high-discrepancy hypergraphs}
\author{J{\'o}zsef Balogh, Danila Cherkashin and Sergei Kiselev}
\date{March 2018}

\author{J\'ozsef Balogh\footnote{Department of Mathematical Sciences, University of Illinois at Urbana-Champaign, IL, USA, and Moscow Institute of Physics and Technology, 9 Institutskiy per., Dolgoprodny, Moscow Region, 141701, Russian Federation.}, 
Danila Cherkashin\footnote{Chebyshev Laboratory, St.~Petersburg State University, 14th Line V.O., 29B, Saint Petersburg 199178 Russia; 
Moscow Institute of Physics and Technology, Lab of advanced combinatorics and network applications, Institutsky lane 9, Dolgoprudny, Moscow region, 141700, Russia; National Research University Higher School of Economics, Soyuza Pechatnikov str., 16, St. Petersburg, Russian
Federation}, 
Sergei Kiselev\footnote{Moscow Institute of Physics and Technology, Lab of advanced combinatorics and network applications, Institutsky lane 9, Dolgoprudny, Moscow region, 141700, Russia.}}

\begin{document}

\maketitle

\begin{abstract}
We suggest a new method on coloring generalized Kneser graphs based on hypergraphs with high discrepancy and small number of edges.
The main result is providing a proper coloring of $K(n, n/2-t,s)$ in $(4+o(1)) (s+t)^2$ colors, which is produced by Hadamard matrices.
Also, we show that for colorings by independent set of a natural type, this result is the best possible up to a multiplicative constant.

Our method extends to Kneser hypergraphs as well.

\end{abstract}

\section{Introduction}

Let $K(n,k,s)$ be the generalized Kneser graph, i.e.~a graph with the vertex set $\binom{[n]}{r}$ and the edges connecting all pairs of vertices with intersection smaller than $s$, where $[n] = \{1, \ldots, n\}$.
Denote by $J(n,k,s)$  the generalized Johnson graph, i.e.~the graph with the same vertex set $\binom{[n]}{r}$ and the edges connecting all pairs of vertices with intersection exactly $s$.

These graphs are quite popular objects in combinatorics. 
The chromatic number of the generalized Kneser graph were studied by Frankl and F{\"u}redi~\cite{frankl1985chromatic,frankl1986extremal} for fixed $k$ and $s$. 
Diameters of $K(n,k,s)$ and  $J(n,k,s)$ are computed in~\cite{chen2008diameter} and~\cite{agong2018girth} respectively.

Bobu and Kupriyanov~\cite{bobu2016chromatic} studied the chromatic number of $J(n,n/2,s)$ for small values of $s$.
They got the following results.
\begin{theorem}
\label{stupid}
For every $s < n/2$ we have
\[
s + 2 \leq \chi \left [J \left (n, \frac{n}{2}, s \right ) \right ] \leq 2\binom{2s+1}{s+1}.
\]
\end{theorem}
The lower bound immediately follows from the classical result of Lov{\'a}sz~\cite{lovasz1978kneser} on the Kneser graphs. It turns out that the upper bound is a particular case of our Lemma~\ref{main} with $t = 0$ and the complete hypergraph
\[
H = \left ( [2s+1], \binom{[2s+1]}{s} \right ).
\]
In this paper we improve the upper bound in Theorem~\ref{stupid} to a quadratic function  in $s$ for  $s = O(\sqrt n)$ using high-discrepancy hypergraphs. 
When $s \geq \sqrt{n}$ then we have the chromatic number grows as $n\cdot e^{(2+o(1))s^2/n}$ which simplifies to $e^{(2+o(1))s^2/n}$
when $n \ln n = o(s^2)$.
The latter again corresponds to discrepancy results (see Section~\ref{more}).

For every subset $A \subset [n]$ size of at least $s$ we define the corresponding independent set in $K (n,r,s)$ (usually it is called Frankl's set):
\[
I_A := \left \{v \in \binom{[n]}{r} \ \left  |  \ |v \cap A| \geq \frac{|A| + s}{2} \right . \right \}. 
\]
The complete intersection theorem of Ahlswede--Khachatryan~\cite{ahlswede1997complete} states that the independence number of $K (n,r,s)$ is always realized on some $I_A$ 
(they give precise formulation what $I_A$ is the biggest depending on $n$, $r$ and $s$).

Now define the \textit{F-chromatic number} as the chromatic number, that uses only coloring by Frankl's sets. 
Obviously, $\chi_{F} \geq \chi$ and
\[
\chi \left [J \left (n, \frac{n}{2}, s \right ) \right ] \leq \chi \left [K \left (n, \frac{n}{2}, s + 1 \right ) \right ].
\]

Our the first main theorem is the following. 
\begin{theorem}
\label{union} The following holds
\begin{itemize}
    \item [(i)] if $s = \sqrt n/2$ then
    \[
    \chi_F \left [K \left (n, \frac{n}{2}, s \right ) \right ] = \Theta (s^2);
    \]
    \item [(ii)] if $\sqrt n/2 \leq  s \leq O(\sqrt{n \ln n})$ then there is a constant $c > 0$ such that
    \[
    n \cdot e^{cs^2/n} \leq \chi_F \left [K \left (n, \frac{n}{2}, s \right ) \right ] \leq n\cdot e^{(2+o(1))s^2/n};
    \]
    \item [(iii)] if $\sqrt{n \ln n} \ll s$ then
    \[
    \chi \left [K \left (n, \frac{n}{2}, s \right ) \right ] = e^{(2+o(1))s^2/n}.
    \]
\end{itemize}
\end{theorem}

The second main result is about generalized Kneser hypergraphs.
Define
\[
KH (n, r, k, s) = \left (\binom{[n]}{k}, \{ \{v_1, \ldots, v_r \} \, \left | \, |v_i \cap v_j| < s \right. \}    \right ).
\]

\begin{theorem}
Let $n \geq m > (r(r-1)(s-1) + rt)^2$, and assume that there is a Hadamard matrix of size $m$. Then
\[
\chi \left [KH \left (n, r, \frac{n}{r}-t, s \right ) \right ] \leq 2m.
\]
\label{hyperhyper}
\end{theorem}

The paper is organized as follows. In Section~\ref{lowerbounds} we prove Theorem~\ref{lower}, which implies the lower bounds in Theorem~\ref{union}~(i) and~(ii). 
In Section~\ref{upper} we prove the upper bound in Theorem~\ref{union}~(i). In Section~\ref{more} we prove Theorem~\ref{union}~(iii) and the upper bound in Theorem~\ref{union}~(ii).
In Section~\ref{Hyper} we prove Theorem~\ref{hyperhyper}. And finally in Section~\ref{applic} we give a geometric application, and we closing the paper with some concluding remarks and open questions.


\section{Lower bound for Theorem~\ref{union}~(i) and~(ii)}
\label{lowerbounds}

The \textit{discrepancy} of a 2-coloring is the maximum 
over all edges
of the difference between the number of vertices of 
the two colors in the edge. 
The \textit{discrepancy} of a hypergraph $H$ is the minimum 
discrepancy of among all 2-colorings of this hypergraph; we denote it by $\disc (H)$.

\begin{lemma}
Consider a proper F-coloring of $K (n, n/2 - t, s)$.
Let the family $\{A_i\}_{i=0}^q$ generate the independent sets used for the proper $F$-coloring.
Set $V = \cup A_i$ and $E = \{A_i\}$. 
Then the hypergraph $H = (V, E)$ has discrepancy at least $s$.
\label{back}
\end{lemma}

\begin{proof}
Consider an arbitrary 2-coloring of $[n]$. Obviously, there is a monochromatic (say, red) vertex $v$ of $K (n, n/2-t, s)$. Since $v$ is colored,
there is an edge $A \in E$ containing at least $(|A|+s)/2$ red elements from $[n]$, so $\disc (H) \geq s$.
\end{proof}

We need the following theorem from~\cite{spencer1985six} (it is also written in~\cite{alon2016probabilistic}).

\begin{theorem}
Let  $H = (V,E)$ be a hypergraph. Then
\[
\disc (H) \leq 12 \sqrt{|E|}.
\]
\label{Spencer}
\end{theorem}

By Lemma~\ref{back}, a proper $F$-coloring of $K (n, n/2-t, s)$ with $q$ colors gives us a hypergraph $H$ with $q$ edges such that $\disc (H) \geq s$.
Then by Theorem~\ref{Spencer}
\[
12\sqrt q \geq \disc (H) \geq s \ \ \   \mbox { so }  \ \ \ q \geq \frac{1}{144}s^2. 
\]

\noindent Thus, we proved the lower bound in the Theorem~\ref{union} and a bit more.

\begin{theorem}
For every $t \leq n/2$ and $s \leq n/2 - t$ we have
\[
\chi_{F} \left [K \left (n, \frac{n}{2}-t, s \right ) \right ] \geq \frac{1}{144}s^2.
\]
\label{lower}
\end{theorem}

The last displayed inequality in Section 12.2 in~\cite{alon2004probabilistic} states that there is a constant $K$ such that for every hypergraph $H = (V,E)$ with $|V| < |E|$  the following holds
\[
\disc (H) \leq K \sqrt{|V| \ln \frac{|E|}{|V|}}.
\]
Now we can prove the lower bound in Theorem~\ref{union}~(ii). If there is an $F$-coloring with $q$ colors then 
\[
K \sqrt{n \ln \frac{q}{n}} \geq \disc (H) \geq s,
\]
and hence
\[
q \leq n\cdot e^{s^2K^{-2}n^{-1}}.
\]

\section{Upper bound in Theorem~\ref{union}~(i)}
\label{upper}

The \textit{$t$-shifted discrepancy} of a 2-coloring is the maximum 
over all edges $e$
of the quantity
\[
|blue (e) - red (e) + t|,
\]
where $blue (e)$, $red (e)$ mean the number of blue and red vertices in $e$.
The \textit{$t$-shifted discrepancy} of a hypergraph $H$ is the minimum 
discrepancy of among all 2-colorings of this hypergraph.

\begin{lemma}
Let $H = (V, E)$ be a hypergraph with the $t$-shifted discrepancy at least $s + t$, and $|V| \leq n$.
Then
\[
\chi_{F} \left [K \left (n, \frac{n}{2}-t, s \right ) \right ] \leq 2|E|.
\]
\label{main}
\end{lemma}

\begin{proof}
Embed $H$ into $[n]$. 
For every edge $e \in E$ define colors $1_e$ and $2_e$ as follows:
\[
1_e := \left \{A \in V\left(K\left(n,\frac{n}{2}-t,s\right) \right ) \left | \  |A \cap e| \geq \frac{|e| + s}{2} \right \} \right. ; 
\]
\[
2_e := \left \{A \in V\left(K\left(n,\frac{n}{2}-t,s\right) \right ) \left | \ |A \cap \bar{e}| \geq \frac{|\bar{e}| + s}{2} \right \} \right. .
\]
Vertices of the same color span an independent set, because every pair of vertices with the same color intersects by at least $s$ points.

Every set $A \subset [n]$ of size $\frac{n}{2} - t$ gives a 2-coloring of $H$ by setting blue color to $V(H) \cap A$
and red color to $V(H) \setminus A$. 
By the condition on the $t$-shifted discrepancy, there is a hyperedge $e\in H$ such that 
\[
\big||A \cap e| - |A \cap \bar e| + t \big| \geq s + t 
\]
which implies
\[|A \cap e| \geq \frac{|e| + s}{2}\quad 
{\text or }  
\quad |A \cap e| \leq \frac{|e| - s - 2t}{2}.\]
It means that $A \in 1_e$ or $A \in 2_e$ respectively, because $|A \cap \bar{e}| \geq \frac{|\bar{e}|+s}{2}$ is equivalent to $|A \cap e| \leq \frac{|e| - (s + 2t)}{2}$.

To summarize: all the vertices are colored, and every color class is an independent set, i.\,e. our coloring is a proper F-coloring.
\end{proof}

\begin{theorem}
Let $n \geq m > 4(s + t)^2$, and assume that there is a Hadamard matrix of size $m$.
Then
\[
\chi_{F} \left [ K \left (n, \frac{n}{2}-t, s \right )  \right ] \leq 2 m.
\]
\label{Hadamard}
\end{theorem}

\begin{proof}
It is well-known that a Hadamard matrix of size $m$ produces a hypergraph with discrepancy at least $\sqrt{m}/2$.
We repeat the proof of this from~\cite{alon2016probabilistic} and show that for any $t$ it has $t$-shifted discrepancy at least $\sqrt{m-1}/2$.

Let $H = \{h_{ij}\}$ be a Hadamard matrix of order $m$ with first row and first column all ones. Any Hadamard matrix
can be so ``normalized'' by multiplying appropriate rows and columns by $-1$. Let
$v = (v_1, \ldots, v_m)$, $v_i = \pm 1$.
Then
\[
Hv = v_1c_1 + \ldots + v_mc_m,
\]
where $c_i$ denotes the $i$-th column vector of $H$. Writing $Hv = (L_1, \ldots L_m)$ and
letting $|c|$ denote the usual Euclidean norm,
\[
L_1^2 + \ldots + L_m^2 = |Hv|^2 = v_1^2|c_1|^2 + \ldots + v_m^2|c_m|^2 = m + \ldots + m = m^2,
\]
since the $c_i$'s are mutually orthogonal.
Note also that
\[
L_1 + \ldots + L_m = \sum_{i,j = 1}^{m} v_jh_{ij} =  \sum_{j = 1}^{m} v_j \sum_{i = 1}^{m} h_{ij} = mv_1 = \pm m.
\]
Let $J$ be the all ones matrix of order $m$.
Set $\lambda = v_1 + \ldots + v_m$  so that $Jv = (\lambda, \ldots, \lambda)$ and
\[
\frac{H + J}{2}v = \left (\frac{L_1 + \lambda}{2}, \ldots, \frac{L_m + \lambda}{2} \right ). 
\]
Let $x = \lambda - 2t$ so 
\begin{equation}
\sum_1^m (L_i + x)^2 = \sum_1^m L_i^2 + 2x\sum_1^m L_i + mx^2 = m^2 \pm 2mx + mx^2 \geq m(m - 1).
\label{mainequality}
\end{equation}
This implies that for some $i$ we have 
\[
\left |\frac{L_i + \lambda}{2} - t \right| \geq \frac{\sqrt{m-1}}{2}.
\]
Lemma~\ref{main} finishes the proof.
\end{proof}

It is worth noting that a random approach gives the same (up to a constant factor) bound, see~\cite{alon2016probabilistic}.

Theorem~\ref{Hadamard} implies the upper bound in Theorem~\ref{union}~(i).

\section{Proof of Theorem~\ref{union}~(iii) and the upper bound in Theorem~\ref{union}~(ii)}
\label{more}

Frankl and Wilson proved~\cite{frankl1981intersection} the following theorem.

\begin{theorem}
Let $s \leq n/4$ and $n/2 - s$ be a primary (prime or a power of a prime) number.
Then
\[
\alpha \left [ J \left (n, \frac{n}{2}, s \right )  \right ] \leq \binom{n}{\frac{n}{2} - s - 1}. 
\]
\end{theorem}

Now we are working in the setup $s^2 = \Omega(n)$. First, suppose that $n/2-s$ is a primary number. Then the lower bound on the chromatic number of 
$K := K(n,n/2,s)$ becomes non-trivial:
\[
\frac{\binom{n}{n/2}}{\binom{n}{n/2-s}} = \frac{(n/2+s)\cdot \ldots \cdot (n/2+1)}{(n/2)\cdot \ldots \cdot  (n/2-s+1)} = \frac{\prod_{i=1}^{s} (1 + \frac{i}{n/2})}{\prod_{i=0}^{s-1} (1 - \frac{i}{n/2})} = 
e^{(2 + o(1))s^2/n}.
\]
The prime numbers are dense, so again
\begin{equation}
\chi \left [ K \left (n, \frac{n}{2}, s+1 \right )  \right ] \geq \chi \left [ J \left (n, \frac{n}{2}, s \right )  \right ] \geq \frac{|V(J)|}{\alpha (J)}  \geq  e^{(2+o(1))s^2/n}. 
\label{generalchi}
\end{equation}

To show that the bound is near to optimal, recall Lov{\' a}sz theorem on the fractional covers~\cite{lovasz1975ratio}. Let $\tau$ be the minimal number of edges required to cover all vertices of $H$; 
$\tau^*$ be the minimal sum of weights on edges required to cover all the vertices of $H$ in such a way that the sum of weights over edges containing every vertex is at least 1. 

\begin{theorem}
Let $H = (V,E)$ be a hypergraph, $d$ be the maximum edge size of $H$. 
Then
\[
\tau (H) \leq (1 + \ln d) \tau^*.
\]
\label{fractional}
\end{theorem}

Now we prove that we have equality in~(\ref{generalchi}).
Let $A \subset [n]$ be such a set that realized $\max_A |I_{A}|$. We shall apply Theorem~\ref{fractional} to the hypergraph 
\[
H := \left ( \binom{[n]}{\frac{n}{2}}, \{\pi (I_A) \ | \  \pi \in S_n \}  \right ),
\]
where $S_n$ is the permutation group over $[n]$.
A collection of hyperedges covering  $V(H)$ provides a  
proper F-coloring of $K$, so
\[
\tau (H) \geq \chi_F(K) \geq \chi (K).
\]
By vertex transitivity we have
\[
\tau^*(H) = \frac{|V(H)|}{d(H)} = \frac{|V(K)|}{\alpha(K)}.
\]
Putting all together
\begin{equation}
\chi_F (K) \leq \tau (H) \leq \tau^*(1 + \ln d(H)) \leq \frac{|V(K)|}{\alpha (K)} (1 + \ln d(H)) \leq n e^{(2 + o(1))s^2/n}. 
\label{star}
\end{equation}
Because of~(\ref{generalchi}) for the case $n \ln n = o(s^2)$ we have 
\[
\chi \left [ K \left (n, \frac{n}{2}, s \right )  \right ] = e^{(2+o(1))s^2/n}.
\]


\section{Hypergraph case}
\label{Hyper}

In this section we are going to extend our methods to generalized Kneser hypergraphs.

It turns out that we can repeat the arguments for graphs in this case.
The \textit{$w$-shifted $r$-centered discrepancy} of a 2-coloring is the maximum 
over all edges
of the quantity
\[
|(r-1) blue (e) - red(e) + w|,
\]
where $blue (e)$ and $red(e)$ mean the number of blue and red vertices in $e$.
The \textit{$r$-centered discrepancy} of a hypergraph $H$ is the minimum 
discrepancy of among all 2-colorings of this hypergraph.

\begin{lemma}
Let $H = (V, E)$ be a hypergraph with $tr/2$-shifted $r$-centered discrepancy $x > \frac{r(r-1)(s-1) + rt}{2}$, and  $|V| \leq n$.
Then
\[
\chi \left [KH \left (n, r, \frac{n}{r}-t, s \right ) \right ] \leq 2|E|.
\]
\label{Hypermain}
\end{lemma}

\begin{proof}
Embed $H$ into $[n]$. 
For every edge $e \in E$ define colors $1_e$ and $2_e$ as follows:
\[
1_e := \left \{A \in V\left(K \left (n, r, \frac{n}{r} - t, s \right ) \right ) \left | \  |A \cap e| > \frac{|e|}{r} + \frac{(r-1)(s-1)}{2} \right \} \right. ; 
\]
\[
2_e := \left \{A \in V\left(K \left (n, r, \frac{n}{r} - t, s \right ) \right ) \left | \ |A \cap \bar{e}| > \frac{|\bar{e}|}{r} + \frac{(r-1)(s-1)}{2} \right \} \right..
\]

First, let us show that every vertex has a color. Every set $A \subset [n]$ of size $\frac{n}{r} - t$ gives a 2-coloring of $H$ by setting blue color to $V(H) \cap A$ and red color to $V(H) \setminus A$. 
By the condition on the $tr/2$-shifted $r$-centered discrepancy, there is a hyperedge $e\in H$ such that 
\[
\left |(r-1)blue(e) - red(e) + \frac{tr}{2} \right| \geq x > \frac{r(r-1)(s-1)}{2} + \frac{tr}{2}.
\]
Using $blue(e) = |A \cap e|$ and $red(e) = |e| - | A \cap e|$ we have either
\[
-|e| + r|A \cap e| + \frac{tr}{2} > \frac{r(r-1)(s-1)}{2} + \frac{tr}{2}
\]
or
\[
-|e| + r|A \cap e| + \frac{tr}{2} < -\frac{r(r-1)(s-1)}{2} - \frac{tr}{2}.
\]
In the first case 
\[
|A \cap e| > \frac{|e|}{r} + \frac{(r-1)(s-1)}{2},
\]
i.\,e. $A$ is colored by $1_e$. In the second case
\[
|A \cap e| < \frac{|e|}{r} - \frac{(r-1)(s-1)}{2} - t.
\]
Using $|A \cap \bar e| + |A \cap e| = \frac{n}{r} - t$ we have
\[
|A \cap \bar e| > \frac{n}{r} - t - \frac{|e|}{r} + \frac{(r-1)(s-1)}{2} + t = \frac{|\bar e|}{r} + \frac{(r-1)(s-1)}{2},
\]
i.\,e. $A$ is colored by $2_e$.

Suppose that there is a monochromatic edge $\{v_1, \ldots v_r\}$.
Then there is an edge or the complement of an edge (denote it by $e$) such that for $i = 1, \ldots, r$ the following holds
\[
|v_i \cap e| > \frac{|e|}{r} + \frac{(r-1)(s-1)}{2}.
\]
This implies
\[
\sum_{i = 1}^{r} |v_i \cap e| > |e| + r\frac{(r-1)(s-1)}{2}.
\]
From the other hand, since $|v_i \cap v_j| < s$ we have
\[
|e| \geq \sum_{i = 1}^{r} |v_i \cap e| - \frac{r(r-1)}{2}(s-1),
\]
a contradiction.

\end{proof}

The proof of Theorem~\ref{hyperhyper} is similar to the proof of Theorem~\ref{Hadamard} with the replacement of Lemma~\ref{main} with Lemma~\ref{Hypermain}.
Since the inequality~(\ref{mainequality}) holds for every $x$ and all the sets from Hadamard matrix has the size $m/2$
it works for centered and shifted discrepancy.

\section{A geometric application}
\label{applic}

First, we need some additional definitions. Let 
\[
V_k := \{v \in \{0, \pm 1\}^n \  | \ |v| = \sqrt{r} \};  \ \ \ \  V_{k,l} := \{v \in \{0, \pm 1\}^n \  | \ v \mbox{ has exactly } r \ \ '1' \mbox{ and exactly } l \ \  '-1' \}.
\]
Furthermore, we define 
\[
K (n,k,s) := (V_k, \{(v_1,v_2) \ | \ (v_1,v_2) < s\} ) ;  \ \ \ \ K (n,k,l,s) := (V_{k,l}, \{(v_1,v_2) \ | \ (v_1,v_2) < s\} );
\]
\[
J (n,k,s) := (V_k, \{(v_1,v_2) \ | \ (v_1,v_2) = s\} ) ;  \ \ \ \ J (n,k,l,s) := (V_{k,l}, \{(v_1,v_2) \ | \ (v_1,v_2) = s\} ).
\]
Also, the \textit{support} of a vector is the set $\supp (v) \in [n]$ of its non-zero coordinates.

Obviously, there is a natural bijection between the subsets of $[n]$ and $\{0,1\}$-vectors in $\mathbb{R}^n$.
If we fix the size of subsets, then the corresponding vectors lie on a sphere. It implies one-to-one correspondence between the scalar product and the Euclidean distance in this case.
So $J(n,k,s)$ is a distance graph.

The independence number of such a graphs were studied~\cite{frankl2016intersection,frankl2017families,frankl2018erdHos} by Frankl and Kupavskii.
In~\cite{frankl2018erdHos} the authors find an explicit value of the independence number of $J (n,k,1,-2)$.
Paper~\cite{frankl2016intersection} is devoted to the independence numbers of $K (n, k, s)$ with $n > n_0 (k, s)$.
Finally, the work~\cite{frankl2017families} deals with the independence number of $J (n,k,l,-2l)$ for arbitrary $n$, $k$ and $l$.

It is also worth noting that Cherkashin, Kulikov and Raigorodskii~\cite{cherkashin2015chromatic} improved lower bounds on the chromatic numbers of small-dimensional Euclidean spaces via chromatic numbers of $J(n,3,1)$.

In this setup $s$ is the difference between the minimal scalar product and the maximum restricted scalar product; $t$ is the difference between $n/2$ and the size of a vertex support.
Using Theorem~\ref{Hadamard} we have that the chromatic number is at most quadratic in the small parameters.

\begin{corollary}
Let $n \geq m > 4(s + 2l + t)^2$, and assume that there is a Hadamard matrix of size $m$.
Then
\[
\chi \left [ K \left (n, \frac{n}{2} - l - t, l, -2l + s \right )  \right ] \leq 2 m.
\]
\end{corollary}

\begin{proof}
Note that if vertices $v_1$, $v_2$ are such that $| \supp (v_1) \cap \supp (v_2) | = q$, then the scalar product $(v_1, v_2) \geq q - 4l$. 
So for $q \geq 2l+s$ the vertices are not adjacent.
Hence
\[
\chi  \left [ K \left (n, \frac{n}{2} - l - t, l, -2l + s \right )  \right ] \leq \chi \left [ K \left (n, \frac{n}{2} -  t, 2l + s \right ) \right ] .
\]
Using Theorem~\ref{Hadamard} we are done.
\end{proof}

\section{Discussion}

\label{discuss}

\paragraph{Improving the lower bound for general colorings.} Recall that the exact value of the chromatic number of Kneser graph was determined by Lov{\'a}sz~\cite{lovasz1978kneser}. Then Alon, Frankl and Lov{\'a}sz~\cite{alon1986chromatic} determined the chromatic number of Kneser hypergraphs, i.\,e. proved that
\[
\chi (KH (n, r, k, 1)) =  \lc \frac{n - r(k - 1)}{r - 1}  \rc.
\]
Since then several different proofs have been appeared.
One of the two main ways to prove the lower bound uses Borsuk--Ulam theorem (or its analogues), see~\cite{ziegler2002generalized, matousek2002chromatic, matouvsek2004combinatorial, matousek2008using} and the other computes
the connectedness of a corresponding complex, see~\cite{lovasz1978kneser,alon1986chromatic,kozlov2007combinatorial}.

Recall that for $k$ being almost $n/r$ we have quadratic dependence on small parameters.
Unfortunately, we are not able to improve the linear lower bound and this problem looks quite challenging.

\paragraph{Discrepancy.} In Section~\ref{more} we showed that there is a close relation between $F$-chromatic number and the discrepancy theory.
For instance, Lemma~\ref{back} and relation~(\ref{star}) imply the existence of a hypergraph $H = (V, E)$ for every $|E| \geq |V|$ such that
\[
\disc (H) \geq (1+o(1)) \sqrt{\frac{|V| \ln \frac{|E|}{|V|}}{2}},
\]
which is optimal up to a constant.

It might be interesting to explore the concept of shifted and centered discrepancy in general.

\paragraph{Constants values of $s$.} Obviously, $\chi [K(n,n/2,1) = J(n,n/2,0)] = \chi_F [K(n,n/2,1) = J(n,n/2,0)] = 2$, because the graph is a matching.
Note that $\disc ([3], \{\{1,2\},\{1,3\}, \{2,3\}\}) = 2$, so by Lemma~\ref{main} we have the following observation
\[
\chi_F \left [ K \left (n, \frac{n}{2}, 2 \right )  \right ] \leq 6.
\]
Also Fano's plane has discrepancy 3, so $\chi_F \left [ K \left (n, \frac{n}{2}, 3 \right )  \right ] \leq 14$, but it seems not optimal.
From the other hand, Theorem~\ref{stupid} gives $3 \leq \chi_F \left [ K (n, n/2, 2 )  \right ]$ and $4 \leq \chi_F \left [ K (n, n/2, 3 )  \right ]$.
Finding the exact values of the chromatic numbers for constant $s$ is also of some interest.

\paragraph{Acknowledgements.} The work was supported by the Russian government grant NSh-6760.2018.1, and by the grant of the Government of the Russian 
Federation for the state support of scientific research carried out under the supervision of leading scientists, agreement 14.W03.31.0030 dated 15.02.2018.
Research of the first author is partially supported by NSF Grant DMS-1500121 and by the Langan Scholar Fund (UIUC).

The authors are grateful to A.~Raigorodskii for the statement of the problem and to H.~R.~Daneshpajouh for drawing their attention to the generalized Kneser hypergraph.

\bibliographystyle{plain}
\bibliography{main}

\end{document}